\documentclass[11pt]{amsart}
\usepackage{amssymb, mathrsfs, bbm}
\usepackage[T1]{fontenc}
\usepackage[sc, osf]{mathpazo}
\selectfont
\usepackage[protrusion=true, expansion=true]{microtype}

\usepackage[all, arc]{xy}
\SelectTips{eu}{}
\entrymodifiers={+!!<0pt,\fontdimen22\textfont2>}

\usepackage[pdftex, colorlinks=true, linkcolor=blue, citecolor=magenta, linktocpage]{hyperref}

\theoremstyle{plain}
\newtheorem{thm}[subsubsection]{Theorem}
\newtheorem{lemma}[subsubsection]{Lemma}

\newtheorem{cor}[subsubsection]{Corollary}
\newtheorem*{mainObs1}{Basic Observation}

\theoremstyle{definition}
\newtheorem{example}[subsubsection]{Example}

\theoremstyle{remark}

\theoremstyle{definition}

\numberwithin{equation}{subsubsection}

\def\cA{\mathcal{A}}

\def\cN{\mathcal{N}}
\def\cO{\mathcal{O}}

\def\11{\mathbf{1}}

\def\CC{\mathbf{C}}

\def\QQ{\mathbf{Q}}

\def\ZZ{\mathbf{Z}}

\def\fn{\mathfrak{n}}

\def\pt{\mathrm{pt}}

\def\Spec{\mathrm{Spec}}

\newcommand{\mapright}[1]{\xrightarrow{#1}}

\newcommand{\const}[1]{\underline{#1}}

\title{Tate classes, equivariant geometry and purity}
\author{R. Virk}
\address{The Appalachians}
\begin{document}
\maketitle
\renewcommand{\thesubsection}{\arabic{subsection}}
\subsection{Introduction}
This note gives a simple explanation for the ubiquity of `vanishing of cohomology in odd degrees' in equivariant contexts. 

The heart of the matter is that actions of linear algebraic groups force the mixed Hodge structures showing up in cohomology to be of type $(n,n)$. Combined with purity, which is also often forced by the equivariant context, yields the aforementioned vanishing.
Arguments of this nature have been standard in Kazhdan-Lusztig theory for decades now (for instance, see \cite{G}, \cite{KL}, \cite{MS}, \cite{So}, \cite{SW}, \cite{Sp}; also see \cite{BJ}). The only ``new contribution"\footnote{
``It is all already in Dedekind".} that this note makes is an insistence on emphasizing the word ``Tate".

Several natural questions, related to algebraic cycles, arise from these observations. These have satisfactory answers, but treating them requires some motivic machinery. This has been relegated to a separate paper.

\textbf{Acknowledgments: }This note is mainly a digression arising from a project joint with W. Soergel and M. Wendt \cite{SVW}. In particular, the Basic Observation below was born out of explanations by W. Soergel of a more general statement for motivic sheaves.

In its current presentation, Theorem \ref{contractingthm} owes its formulation to a conversation with M. A. de Cataldo. 

Finally, this note would never have seen light of day were it not for J. Gandini and A. Maffei's insistence (and constant encouragement) that these results were not completely frivolous.

\subsection{Conventions}\label{s:conventions}
A `variety' will always mean a `separated scheme of finite type over $\mathrm{Spec}(\CC)$'. Hereon, I will write
`$\pt$' instead of `$\Spec(\CC)$'.
Constructible sheaves, cohomology, etc., will always be with $\QQ$-coefficients, and with respect to the complex analytic site associated to a variety.
I will freely use the existence of functorial mixed Hodge structures on cohomology, compactly supported cohomology, equivariant cohomology, etc. (see \cite{D} or \cite{Sa}).
A $\ZZ$-graded mixed Hodge structure $H^*$ (for instance, the cohomology $H^*(X)$ of a variety $X$) will be called \emph{pure} if each $H^i$ is a pure Hodge structure of weight $i$.
A mixed Hodge structure will be called \emph{Tate} if it is an extension of Hodge structures of type $(n,n)$ (the $n$ is allowed to vary).

\subsection{The Basic Observation}\label{s:basic}
The following Lemma is well known.
\begin{lemma}[{\cite[\S 9.1]{D}}]Let $G$ be a linear algebraic group. Then $H^*(G)$ is Tate.
\end{lemma}
\begin{proof}
We may assume $G$ is connected reductive (Levi decomposition). Then the splitting principle applies.
\end{proof}
The following is essentially contained in \cite[\S 9]{D} (in \cite{D} the observation is only made explicit for $X=\pt$; regardless, it is in there).
\begin{mainObs1}Let $G$ be a linear algebraic group acting on a variety $X$. If $H^*(X)$ is Tate, then the $G$-equivariant cohomology $H^*_G(X)$ is Tate.
\end{mainObs1}
\begin{proof}Consider the usual simplicial variety $[X/G]_{\bullet}$ (see \cite[\S6.1]{D}) computing $H^*_G(X)$.
Filtering $[X/G]_{\bullet}$ by skeleta yields a spectral sequence converging to $H^*_G(X)$ \cite[Proposition 8.3.5]{D}. The $E_1$ entries of this spectral sequence are of the form $H^q(G^{\times p}\times X)$. 
\end{proof}
The following consequence is essentially contained in \cite[Theorem 1(a)]{BP} (the argument in \cite{BP} is quite different though, and as formulated, \cite[Theorem 1(a)]{BP} is a statement about $E$-polynomials).
\begin{cor}
Let $G$ be a linear algebraic group and $K\subset G$ a closed subgroup. Then $H^*(G/K)$ is Tate.
\end{cor}
%
%
\begin{cor}\label{tatecor}
Let $X$ be a variety on which a linear algebraic group $G$ acts with finitely many orbits. Then the compactly supported cohomology $H^*_c(X)$ is Tate.
\end{cor}

\subsection{Intersection cohomology}
Write $IH^*(X)$ for the intersection cohomology of $X$, normalized so that if $X$ is smooth and equidimensional, then $IH^k(X) = H^k(X)$. 
\begin{thm}\label{oddvanish}
Let $X$ be a complete variety endowed with the action of a linear algebraic group $G$. If $X$ admits a $G$-equivariant resolution of singularities $E \to X$ such that $E$ admits finitely many orbits, then $IH^*(X)$ vanishes in odd degrees.
\end{thm}

\begin{proof}
By the Decomposition Theorem, $IH^*(X)$ is a direct summand of $H^*(E)$. The latter is pure and Tate (Corollary \ref{tatecor}).
\end{proof}

\subsection{Springer's Homotopy Lemma}\label{obs2}
It will now be convenient to use the language of mixed Hodge modules \cite{Sa}. 
One can avoid this and only use classical (as in the style of \cite{D}) mixed Hodge theory, but this would make the language cumbersome.

Functors on mixed Hodge modules will tacitly be derived. Write $\const{\pt}$ for the trivial (weight $0$) rank one pure Hodge structure on $\pt$. Let $X$ be a variety, and $a\colon X \to \pt$ the structure map. Set 
$\const{X} = a^*\const{\pt}$.

Let $S$ be a variety endowed with a $\CC^{\times}$-action that contracts $S$ to some point $i\colon \{x\} \hookrightarrow S$. Let $a\colon S \to \{x\}$ be the evident map. Call a complex $\cA$, of mixed Hodge modules on $S$, \emph{na\"ively equivariant} if there exists an isomorphism $\alpha^*\cA \simeq p^*\cA$, where $\alpha, p \colon \CC^{\times} \times S \to S$ are the action and projection maps respectively.
\begin{lemma}[Springer's Homotopy Lemma {\cite[Proposition 1]{Sp}}]
If $\cA$ is na\"ively equivariant on $S$, then the canonical maps $a_*\cA \mapright{\sim} i^*\cA$ and $i^!\cA\mapright{\sim}a_!\cA$ are isomorphisms.
\end{lemma}

\subsection{Contracting slices}\label{s:contracting}Let $G$ be a linear algebraic group acting on a variety $X$. A \emph{contracting slice} at a point $x\in X$ is the data of a locally closed subvariety $S\subset X$ containing $x$, and satisfying:
\begin{enumerate} 
\item the map $G\times S \to X$, $(g,x)\mapsto gx$ is smooth;
\item there exists a one parameter subgroup $\CC^{\times}\to G$ that leaves $S$ stable and contracts $S$ to $x$.
\end{enumerate}
We will say that the $G$-action on $X$ \emph{admits contracting slices} if each $G$-orbit contains a point that admits a contracting slice.
The following result is contained either implicitly or explicitly (sometimes in special cases) or in slightly different language (for instance, stalkwise/pointwise purity of pure Hodge modules vs. purity of fibres) in \cite[\S5.2]{BeBe}, \cite{BJ}, {\cite[\S14]{BL}}, \cite{dCMM}, \cite{G}, \cite{KL}, \cite{MS}, \cite{So}, \cite{SW}, \cite{Sp}. Undoubtedly, this is an incomplete list: the use of contracting slices pervades representation theory.
\begin{thm}
\label{contractingthm}Let $G$ be a linear algebraic group acting on $E$ and $X$. Assume $E$ is rationally smooth, and admits finitely many orbits. Let $\pi\colon E\to X$ be a $G$-equivariant proper morphism. If $X$ admits contracting slices, then the cohomology of each fibre $H^*(\pi^{-1}(x))$, $x\in X$, is pure and Tate. In particular, $H^*(\pi^{-1}(x))$ vanishes in odd degrees.
\end{thm}

\begin{proof}
The purity assertion is a special case of the well known fact that contracting slices guarantee pointwise purity (i.e., purity of stalks and costalks at all points) of every pure $G$-equivariant Hodge module on $X$. In slightly more detail, it suffices to prove the result for a single point in each $G$-orbit in $X$. The restriction of $\pi_*\const{E}$ to a contracting slice is pure \cite[\S2.3.2]{MS}. Thus, Springer's Homotopy Lemma yields purity at each contraction point. As there are only finitely many $G$-orbits in $E$, the isotropy group $G_x$ acts with finitely many orbits on the fibre $\pi^{-1}(x)$. So Corollary \ref{tatecor} applies.
\end{proof}
%
%
%
\subsection{Examples}\label{s:examples}
The above story now applies to flag varieties, toric varieties, symmetric varieties, wonderful compactifications, ..., where there are well known group actions and/or contracting slices.

\begin{example}Let $X$ be a complete rationally smooth variety on which a linear algebraic group $G$ acts with finitely many orbits (for instance a complete simplicial toric variety). Then $H^*(X)$ is pure (rational smoothness plus completeness) and Tate (Corollary \ref{tatecor}). In particular, $H^*(X)$ vanishes in odd degrees.
\end{example}

\begin{example}Let $G$ be a connected reductive group. A $G$-variety is called \emph{spherical} if it contains a dense orbit for a Borel subgroup of $G$.
Complete spherical varieties are known to satisfy the assumptions of Theorem \ref{oddvanish}. In particular, if $X$ is spherical, then $IH^*(X)$ vanishes in odd degrees \cite{BJ}. Note that toric varieties are spherical.
\end{example}

\begin{example}Let $B\subset G$ be a Borel subgroup, and $\pi\colon E \to G/B$ a $B$-equivariant proper morphism to the flag variety $G/B$. Assume $E$ is rationally smooth, and admits finitely many orbits. Let $x\in G/B$. Then, for a suitable product $U$ of root subgroups of $G$, the map $u \mapsto ux$ defines an embedding $U\hookrightarrow G/B$ whose image is a cell transversal to the $B$-orbit of $x$. This cell is contracted by a one parameter subgroup (of $B$) to $x$. Consequently, Theorem \ref{contractingthm} applies, and $H^*(\pi^{-1}(x))$ vanishes in odd degrees for all $x\in G/B$. In fact, $H^*(\pi^{-1}(x))$ is generated by algebraic cycles (this follows from purity combined with \cite[Theorem 3]{To}). This generalizes the fact that fibres of Bott-Samelson resolutions can be paved by affine spaces.
\end{example}

\begin{example}The same result as in the previous example holds if we replace $G/B$ by $G/K$, where $K\subset G$ is a symmetric subgroup. Contracting slices are known to exist for the $B$-action on $G/K$ \cite{MS}.
\end{example}

\begin{example}Analogously, Theorem \ref{contractingthm} applies to spherical varieties that admit contracting slices for the $G$-action. Not all spherical varieties admit contracting slices. Regardless, it can be shown (use the argument in the proof of \cite[Theorem 4]{BJ}) that on a \emph{normal} spherical variety, every pure $G$-equivariant mixed Hodge module is pointwise pure (i.e., its stalks and costalks are pure at every point). This immediately yields the conclusions of Theorem \ref{contractingthm}.
\end{example}

\begin{example}Let $G$ be connected semisimple, and let $\cN$ be the cone of nilpotent elements in $Lie(G)$. Then the adjoint action of $G$ on $\cN$ admits contracting slices (for instance, see \cite[\S3.7.14]{CG}). Hence, Theorem \ref{contractingthm} applies. Unhappily, this doesn't yield the vanishing of the cohomology of Springer fibres in odd degrees, since $G$ doesn't act on the Springer resolution with finitely many orbits.
\end{example}

\subsection{Complements}
\begin{enumerate}
\item The Basic Observation is a specific instance of the more general observation that if $X_{\bullet}$ is a simplicial variety with each $H^*(X_i)$ Tate, then $H^*(X_{\bullet})$ is Tate. This yields statements, analogous to the Basic Observation, for algebraic stacks with atlases.

I don't know any examples (apart from $[X/G]$) where this yields anything interesting that is not already well known using simpler methods. However, see \cite{Sh}.
\item The Basic Observation has a weak converse: if $H^*(X)$ is pure, and $H^*_G(X)$ is Tate, then $H^*(X)$ is Tate. This is immediate, since purity implies $H^*_G(X)\simeq H^*_G(\pt) \otimes H^*(X)$ as an $H^*_G(\pt)$-module.
I don't know of a counterexample to this statement with the purity assumption dropped.
\item Springer's Homotopy Lemma uses $\CC^{\times}$-actions to infer purity. These can also be exploited to deduce the Tate property:

Let $X$ be a variety endowed with a $\CC^{\times}$-action. Assume $H^*(X)$ is pure. If the cohomology of the fixed point subvariety $H^*(X^{\CC^{\times}})$ is Tate, then so is $H^*(X)$. 

This assertion should be viewed as a cohomological counterpart to the classical Bialynicki-Birula decomposition \cite{BB}.
To prove it, note that the Localization Theorem (in equivariant cohomology) yields that restriction $H^*_{\CC^{\times}}(X) \to H^*_{\CC^{\times}}(X^{\CC^{\times}})$ is an isomorphism modulo $H^*_{\CC^{\times}}(\pt)$-torsion. Purity of $H^*(X)$ implies:
\[ H^*_{\CC^{\times}}(X) \simeq H^*_{\CC^{\times}}(\pt) \otimes H^*(X) \]
as an $H^*_{\CC^{\times}}(\pt)$-module. In particular, $H^*_{\CC^{\times}}(X)$ is free. Consequently, the restriction $H^*_{\CC^{\times}}(X) \hookrightarrow H^*_{\CC^{\times}}(X^{\CC^{\times}})$ is an injection. Both $H^*_{\CC^{\times}}(X^{\CC^{\times}})$ and $H^*_{\CC^{\times}}(\pt)$ are Tate. Therefore, $H^*(X)$ must also be Tate.
\item Although I have not checked the details, everything in this note should extend readily to the context of P. Deligne's Weil conjecture machinery and Frobenius actions on $\ell$-adic cohomology.
\end{enumerate}

\end{document}